\documentclass[12pt]{amsart}

\usepackage{amsthm,amsfonts,amsmath,amssymb,latexsym,epsfig}
\usepackage[usenames]{color}

\DeclareMathAlphabet\oldmathcal{OMS}        {cmsy}{b}{n}
\SetMathAlphabet    \oldmathcal{normal}{OMS}{cmsy}{m}{n}
\DeclareMathAlphabet\oldmathbcal{OMS}       {cmsy}{b}{n}
 
\usepackage{eucal}

\usepackage[active]{srcltx}

\newtheorem{thm}{Theorem}
\newtheorem{lem}[thm]{Lemma}
\newtheorem{prop}[thm]{Proposition}
\newtheorem{cor}[thm]{Corollary}
\newtheorem{defn}[thm]{Definition}
\newtheorem{rem}{Remark}
\newtheorem{question}[thm]{Question}

\newenvironment{rmk}{\refstepcounter{thm} \medskip \noindent {\bf  Remark \arabic{thm}.\,}}{\hfill\mbox{}\bigskip}
\newtheorem*{ack}{Acknowledgements}

\def\d{\partial}

\def\<{\langle}

\def\>{\rangle}

\def\fract#1#2{\raise4pt\hbox{$ #1 \atop #2 $}}

\def\bbc{{\mathbb C}}

\def\bbq{{\mathbb Q}}
\def\bbr{{\mathbb R}}

\def\bbz{{\mathbb Z}}

\def\gra{\alpha}

\def\grd{\delta}

\def\gri{\iota}

\def\grl{\lambda}

\def\grz{\zeta}

\def\grS{\Sigma}

\def\bfw{{\bf w}}

\def\calo{{\mathcal O}}

\def\cald{{\mathcal D}}

\def\calf{{\mathcal F}}

\def\cali{{\mathcal I}}

\def\calo{{\mathcal O}}

\def\cals{{\oldmathcal S}}

\def\calx{{\mathcal X}}

\def\la#1{\hbox to #1pc{\leftarrowfill}}
\def\ra#1{\hbox to #1pc{\rightarrowfill}}

\def\ge{{\mathfrak e}}

\def\gt{{\mathfrak t}}
\def\gu{{\mathfrak u}}

\def\gA{{\mathfrak A}}

\def\gF{{\mathfrak F}}

\def\gM{{\mathfrak M}}

\newcommand{\cH}{\mathcal{H}}

\newcommand{\cL}{\mathcal{L}}

\newcommand{\cO}{\mathcal{O}}

\newcommand{\cS}{\mathcal{S}}

\newcommand{\cW}{\mathcal{W}}

\def\C{\mathbb{C}}

\def\R{\mathbb{R}}

\newcommand{\Lie}{\operatorname{Lie}}

\newcommand{\Aut}{\operatorname{\mathfrak{Aut}}}

\newcommand{\Cal}{\operatorname{Cal}}

\newcommand{\Fut}{\operatorname{Fut}}

\newcommand{\ol}[1]{\overline{#1}}

\def\hook{\mathbin{\hbox to 6pt{%
                 \vrule height0.4pt width5pt depth0pt
                 \kern-.4pt
                 \vrule height6pt width0.4pt depth0pt\hss}}}

\begin{document}

\title[Relative K-stability]{Relative K-stability and Extremal Sasaki metrics}

\author{Charles P. Boyer and Craig van Coevering}\thanks{The first author was partially supported by grant \#245002 from the Simons Foundation.}
\address{Charles P. Boyer, Department of Mathematics and Statistics,
University of New Mexico, Albuquerque, NM 87131.}
\email{cboyer@math.unm.edu} 
\address{Craig van Coevering, School of Mathematical Sciences, U.S.T.C., Anhui, Hefei 230026, P. R. China}
\email{craigvan@ustc.edu.cn}
\keywords{Sasakian, K-stability, Sasaki-extremal}
\subjclass{53C25 primary, 32W20 secondary}

\begin{abstract}
We define K-stability of a polarized Sasakian manifold relative to a maximal torus of automorphisms.   The existence of a 
Sasaki-extremal metric in the polarization is shown to imply that the polarization is K-semistable.   
Computing this invariant for the deformation to the normal cone gives an extention of the Lichnerowicz obstruction, due
to Gauntlett, Martelli, Sparks, and Yau, to an obstruction of Sasaki-extremal metrics.  We use this to
give a list of examples of Sasakian manifolds whose Sasaki cone contains no extremal representatives. These give the first examples of Sasaki cones of dimension greater than one that contain no extremal Sasaki metrics whatsoever. In the process we compute the unreduced Sasaki cone for an arbitrary smooth link of a weighted homogeneous polynomial.

\end{abstract}

\maketitle

\section{Introduction}\label{sec:intro}

Sasaki-extremal metrics were introduced in \cite{BGS06} and provide a more general 
notion of a canonical metric than constant scalar curvature, which is obstructed by the Futaki invaraiant.   
On a Sasakian manifold it is also natural to deform the Reeb vector field by Hamiltonian holomorphic vector fields,
giving the notion of the Sasaki cone $\gt^+$ of a Sasakian manifold, a notion analogous to the K\"{a}hler cone in K\"{a}hler geometry.
It is natural to ask which Reeb vector fields $\xi\in\gt^+$ admit Sasaki-extremal representatives.  
This subset, called the {\it extremal set}, $\ge\subseteq\gt^+$, was shown to be open by the above authors.  
But little is known about this set because there are no known obstructions, besides the obstructions of Gauntlett, Martelli, Sparks, and Yau~\cite{GMSY06}
which obstruct the existence of a Sasaki-Einstein metric in the positive definite case and thus a Sasaki-extremal metric if the 
Futaki invariant vanishes.  

It is natural to conjecture that the existence of a Sasaki-extremal metric is equivalent to some form of GIT stability
of the affine cone $(Y,\xi)$, where $Y=C(M)\cup\{o\}$ is the K\"{a}hler cone on $M$ with the vertex added, polarized by the Reeb
vector field $\xi$.   Interesting work has been done in this direction~\cite{CoSz12,CoSz15} using K-stability in particular.  
Collins and Sz\'ekelyhidi~\cite{CoSz12} extended K-stability to irregular Sasakian manifolds using the Hilbert series of $(Y,\xi)$, and they
proved that the existence of a constant scalar curvature metric implies K-semistability.  

We define a relative version of the K-stability of Collins and Sz\'ekelyhidi to get an obstruction to Sasaki-extremal metrics.
This is K-stability relative to a maximal torus $T\subset\gA\gu\gt(Y,\xi)$.  We prove that the existence of a Sasaki-extremal structure 
compatible with a polarized cone $(Y,\xi)$ implies that $(Y,\xi)$ is K-semistable relative to a maximal torus $T$.  

Considering the Rees algebra degenerations which, as shown in~\cite{CoSz12}, give the Lichnerowicz obstructions of 
Gauntlett, Martelli, Sparks, and Yau, we give a generalized Lichnerowicz obstruction to the existence of 
a Sasaki-extremal metric depending on a $T$-homogenous $f\in H^0(Y,\calo_Y)$. 
We apply our construction to a large class of weighted homogeneous polynomials by giving a sufficient condition for the obstruction of extremal Sasaki 
metrics in the entire Sasaki cone.  This is Theorem \ref{whpobs} below.
We also present a list of Sasaki manifolds in Table \ref{sasconetable} which has families of Sasaki cones of dimension greater than one with no extremal Sasaki metrics. 
In particular, our results can be rephrased in terms of the moduli space of positive Sasaki classes with vanishing first Chern class, viz.

\begin{thm}\label{infcon}
Let $M$ be one of the smooth manifolds listed in the first five entries or the last entry of Table \ref{sasconetable}. Then the moduli space $\gM^c_{+,0}(M)$ of positive 
Sasaki classes with $c_1(\cald)=0$ has a countably infinite number of components of dimension greater than one and which contain no extremal Sasaki metrics, i.e. $\ge=\emptyset$. 
Moreover, these different components correspond to isomorphic transverse holomorphic structures. 
\end{thm}

\begin{ack}
This paper was born out of discussions that took place at the LeBrunfest sponsored by the Centre de Recherches Math\'ematique (CRM) and  hosted by the Universit\'e du Qu\'ebec \`a Montr\'eal (UQAM) whom we thank for their hospitality and support. We would also like to thank Tristan Collins, Hongnian Huang, and Gabor Sz\'ekelyhidi for comments and their interest in our work.
\end{ack}

\section{Background}
In this section we present a brief discussion of Sasakian geometry and refer to \cite{BG05} for details.
A Riemannian manifold $(M,g)$ is {\it Sasakian} if the metric cone $(C(M),\bar{g}),\ C(M) =M\times\bbr^+, \bar{g}=dr^2 +r^2 g$, is K\"{a}hler. 
A Sasakian structure is a special case of a {\it contact metric structure} $\cals=(\xi,\eta,\Phi,g)$ which is a quadruple where $\cald=\ker\eta$ is a contact structure, 
$\xi$ is the Reeb vector field of the contact form $\eta$, $\Phi$ is an endomorphism field that annihilates $\xi$ and such that $(\cald,\Phi |_\cald)$ is a strictly 
pseudoconvex almost CR structure, and $g$ is a compatible Riemannian metric, i.e. $g(\Phi X,\Phi Y)=g(X,Y)-\eta(X)\eta(Y)$ for any vector fields $X,Y$. 
Then the structure $\cals=(\xi,\eta,\Phi,g)$ is {\it Sasakian} if $\pounds_\xi\Phi=0$ and the almost CR structure $(\cald,\Phi |_\cald)$ is integrable. 
A contact structure $\cald$ is said to be of {\it Sasaki type} if there exists such a quadruple $(\xi,\eta,\Phi,g)$ which is Sasakian with $\cald=\ker\eta$. 
Note that all Sasaki manifolds are oriented and here we shall only consider compact Sasaki manifolds.

Since the Reeb vector field $\xi$ is nowhere vanishing it defines a 1-dimensional foliation $\calf_\xi$ known as the {\it characteristic foliation} which is a 
{\it Riemannian foliation} when $(\xi,\eta,\Phi,g)$ is Sasakian. In fact, in the Sasakian case, the transverse geometry is K\"ahlerian. The {\it basic first Chern class} 
$c_1(\calf_\xi)$ is an important invariant of a Sasakian structure. A Sasakian structure $(\xi,\eta,\Phi,g)$ is {\it of positive type} or simply {\it positive} if $c_1(\calf_\xi)$ 
can be represented by a positive definite basic $(1,1)$ form. We denote the {\it space of Sasakian structures} with Reeb vector field $\xi$ and transverse holomorphic 
structure $\ol{J}$ by $\cS(\xi,\ol{J})$. Note that $c_1(\calf_\xi)$ and hence positivity depends only on the space $\cS(\xi,\ol{J})$ and not on a particular element 
$(\xi,\eta,\Phi,g)\in \cS(\xi,\ol{J})$. So we can think of $c_1(\calf_\xi)$ as representing the {\it Sasaki class} of $\cals=(\xi,\eta,\Phi,g)$. 
By the transverse version~\cite{ElK} of Yau's theorem a positive Sasakian structure has a Sasaki metric of positive Ricci curvature. Recall that in the basic cohomology 
exact sequence
$$0\ra{2.0}H^0_B(\calf_\xi)\fract{\grd}{\ra{2.0}} H^2_B(\calf_\xi)\fract{\gri_*}{\ra{2.0}} H^2(M,\bbr)\ra{2.0}\cdots $$
we have $\gri_*c_1(\calf_\xi)=c_1(\cald)$. Thus, if $c_1(\cald)=0$ there exists $a\in\bbr$ such that $c_1(\calf_{\xi})=a[d\eta]_B$. Following~\cite{BMvK15} 
we denote by $\gM^c_{+,0}(M)$ the {\it moduli space of positive Sasaki classes on $M$ with $c_1(\cald)=0$}. Clearly, in the positive case we have $a>0$. 

If the type of a Sasakian structure on $M$ is positive or indefinite it is possible to deform the characteristic foliation through Sasakian structures by deforming the Reeb vector field. 
This gives rise to the (reduced) {\it Sasaki cone} \cite{BGS06} of $M$. However, here, as in \cite{CoSz12}, it is more convenient to work on the K\"ahler cone $C(M)$ 
or better yet on the K\"{a}hler cone with the vertex added $Y=C(M)\cup\{o\}$ which uniquely has the structure of a normal affine variety.  
We will see that the Reeb vector field $\xi$ gives a notion of a polarization on $Y$.   Choose a torus $T\subset\Aut(Y,\xi)$, which we assume to be maximal, so that  $\xi\in\gt=\Lie(T)$.
We say that a K\"{a}hler cone metric on $Y$ is compatible with $\zeta\in\gt$ if $\zeta$ is its Reeb vector field.   Of course the link $M$ then has a Sasakian structure with 
Reeb vector field $\zeta$. We define the \emph{Sasaki cone} to be the subset $\gt^+ \subset\gt$ 
\begin{equation}\label{eq:Sas-cone1}
\gt^+ =\{ \zeta\in\gt\ |\ \exists\text{ a K\"{a}hler cone metric compatible with }(Y,\zeta) \}
\end{equation}

Given a polarized affine scheme $(Y,\xi)$ with a torus action $T\subset\gA\gu\gt(Y)$ where we assume $\xi\in\gt=\Lie(T)$, the ring of global functions
$\cH=H^0(Y,\cO_Y)$ has a weight space decomposition 
\[ \cH =\bigoplus_{\alpha\in\cW} \cH_{\alpha} \]
where $\cW\subset\gt^*$ is the set of weights.

If $(Y,\xi)$ has a single singular point
it turns out that $\xi\in\gt$ is a Reeb vector field for some Sasakian structure with affine cone $Y$ if and only if for all $\alpha\in\cW, \alpha\neq 0$ we have  $\alpha(\xi)>0$.   
Following \cite{CoSz12} we give a second definition of the \emph{Sasaki cone} (where it is called the \emph{Reeb cone})  in $\gt$ 
\begin{equation}\label{eq:Sas-cone2}
\gt^+ =\{ \xi\in\gt\ |\ \forall~ \alpha\in\cW, \alpha\neq 0 \text{ we have } \alpha(\xi)>0 \}\subset\gt.
\end{equation}

\begin{prop}
The Sasaki cones as defined in (\ref{eq:Sas-cone1}) and in (\ref{eq:Sas-cone2}) are identical.   Thus $\gt^+ \subset\gt$ is an open rational polyhedral cone.
\end{prop}
\begin{proof}
As described in~\cite{vCo11} by the second author, or in~\cite{CoSz12}, by choosing enough homogenous elements $f_i \in\cH$ one can define a $T$-equivariant
embedding $\Psi =(f_1,\ldots, f_N) : Y \rightarrow\C^N$, where $T$ acts diagonally on $\C^N$ with weight $\alpha_j$ acting on $z_j$.  
If $\zeta\in\gt^+$ according to definition (\ref{eq:Sas-cone2})
then the polarized affine cone $(Y,\zeta)$ has a Sasakian structure induced by the K\"{a}hler cone structure on $\C^N$ with Reeb vector field
$\sum_{j=1}^N \alpha_i(\zeta) (x_j \d_{y_j} -y_j \d_{x_j})$,  where $z_j =x_j +iy_j$.   The Sasakian metric on $S^{2N-1}$ is the weighted Sasakian
structure with weights $(w_1,\ldots, w_N) =(\alpha_1(\zeta),\ldots,\alpha_N(\zeta))$, as defined in~\cite{BG05}.

The other direction, that if $\zeta\in\gt$ satisfies definition (\ref{eq:Sas-cone1}) then it satisfies (\ref{eq:Sas-cone2}), is proved in~\cite{CoSz12}.
\end{proof}

Generally, the type of a Sasakian structure can change as one moves in the Sasaki cone; however, this does not happen if $c_1(\cald)=0$.
\begin{lem}\label{poscone}
Let $(M,\cald)$ be a contact manifold of Sasaki type with $c_1(\cald)=0$, and let $\cals=(\xi,\eta,\Phi,g)$ be a Sasakian structure satisfying $\eta=\ker\cald$. 
Suppose further that the Sasaki cone $\gt^+$ has dimension greater than one. Then all Sasakian structures whose Reeb vector field $\xi'$ is in $\gt^+$ are positive.
\end{lem}

\begin{proof}
Let $\cals'=(\xi',\eta',\Phi',g')$ be any Sasakian structure in $\gt^+$. Then as stated above there exists $a\in\bbr$ such that $c_1(\calf_{\xi'})=a[d\eta']_B$. 
But since the dimension of $\gt^+$ is greater than one, $a>0$ by Proposition 8.2.15 of \cite{BG05}. So $\cals'$ is positive. 
\end{proof}

\section{Relative K-stability}
Relative K-stability and its relation with extremal K\"ahler metrics were first studied by Sz\'ekelyhidi \cite{Sze06,Sze07}. Here following Collins and Sz\'ekelyhidi~\cite{CoSz12} 
we adopt these ideas to the Sasaki setting. They defined the {\it $T$-invariant index character} as a Hilbert series for $\xi\in \gt^+$, viz.
\begin{equation}
F(\xi,t) := \sum_{\alpha\in\cW} e^{-t\alpha(\xi) }\dim\cH_{\alpha}.
\end{equation}
It is proved in~\cite{CoSz12} that $F(\xi,t)$ converges and has a meromorphic extension with expression
\begin{equation}\label{a0eqn}
F(\xi,t) =\frac{a_0(\xi) n!}{t^{n+1}} +\frac{a_1(\xi)(n-1)!}{t^n} +O(t^{1-n}).
\end{equation}

Given $\zeta\in\gt$ we define the weight characters:
\begin{align*}
C_{\zeta}(\xi,t) = \sum_{\alpha\in\cW} e^{-t\alpha(\xi)} \alpha(\zeta) \dim\cH_{\alpha}\\
C_{\zeta^2}(\xi,t) = \sum_{\alpha\in\cW} e^{-t\alpha(\xi)}\bigl( \alpha(\zeta)\bigr)^2 \dim\cH_{\alpha}.
\end{align*}
Again, these have meromorphic extensions
\begin{align}
C_{\zeta}(\xi,t) =  \frac{b_0(\xi) (n+1)!}{t^{n+2}} +\frac{b_1(\xi)n!}{t^{n+1}} +O(t^n), \\
C_{\zeta^2}(\xi,t) =\frac{c_0(\xi)(n+2)!}{t^{n+3}} +O(t^{-n-2}).
\end{align}
The coeficients $a_0, a_1, b_0, b_1, c_0$ depend smoothly on $\xi\in\gt^+$ and
\begin{equation}\label{eq:coef-der}
 \begin{split}
b_0 & =\frac{-1}{n+1} D_{\zeta} a_0 \\
b_1 & =\frac{-1}{n}D_{\zeta} a_1 \\
c_0 & =\frac{1}{(n+1)(n+2)}D^2_{\zeta} a_0 .
\end{split} 
\end{equation}
We also have a norm defined on $\gt$ defined by
\begin{equation}\label{eq:norm}
\|\zeta\|_{\xi}^2 =c_0(\xi) -\frac{b_0(\xi)^2}{a_0(\xi)}.
\end{equation}

A \emph{$T$-equivariant test configuration} is a flat family of affine schemes
\[Y=Y_1 \subset\Upsilon\overset{\varpi}{\longrightarrow}\C ,\]
such that $\varpi$ is $\C^*$-equivariant and $T$ acts on the fibers.  The $\C^*$ action induces an action on the central fiber $Y_0$.
Let $\zeta\in\Lie(\C^*)$ be the generator of this action.  The \emph{Donaldson-Futaki invariant} of the test configuration is 
\[ \Fut(Y_0, \xi,\zeta) =\frac{a_1(\xi)}{a_0(\xi)}b_0(\xi) -b_1(\xi).\]
 
In the following we assume that $T\subset\gA\gu\gt(Y)$ is a maximal torus with $\xi\in\gt =\Lie(T)$.  
Define $\chi\in\gt$ to be the element dual to $\Fut(Y_0, \xi,\cdot):\gt\rightarrow\R$ with respect to (\ref{eq:norm}).   
Notice that both the bilinear form $\langle\cdot,\cdot\rangle_{\xi}$ defined by (\ref{eq:norm}) and $\Fut(Y_0, \xi,\cdot)$ vanish on $\R\xi$,
so $\chi$ is well defined modulo $\xi$.  Up to a constant $\chi$ is just the transversely 
extremal vector field, as defined in~\cite{ FuMa95}.   We define the \emph{Donaldson-Futaki invariant relative to $T$} of a test configuration
\[  Fut_{\chi}(Y_0, \xi,\zeta)= Fut(Y_0, \xi,\zeta) -\langle \zeta,\chi\rangle.\]

\begin{defn}
A polarized affine variety  $(Y,\xi)$ with a unique singular point is \emph{K-semistable relative to $T$} if for every T-equivariant test configuration
\[ \Fut_{\chi}(Y_0,\xi,\zeta)\geq 0.\]
\end{defn}

Recall that for any $(\eta,\xi,\Phi,g) \in\cS(\xi,\ol{J})$ the \emph{Calabi functional} is defined by
\begin{equation}
\Cal_{\xi,\ol{J}}(g) =\Bigl( \int_M (S_g -\ol{S}_g)^2 \, d\mu_g \Bigr)^{1/2},
\end{equation}
where $S_g$ is the scalar curvature and $\ol{S}_g$ is the average scalar curvature. A Sasaki metric $g$ is {\it extremal} if it is a critical point of $\Cal_{\xi,\ol{J}}^2$. 
This amounts to the fact that the $(1,0)$ component of the gradient of the scalar curvature $S_g$ is a transversely holomorphic vector field. Equivalently, 
the transverse K\"ahler structure is extremal.   Note that when $g$ is extremal, this transversely holomorphic vector field is, up to a constant factor, $\chi$ defined above.

An interesting property of the Calabi functional follows from the Donaldson lower bound~\cite{Don05} extended to Sasakian manifolds in \cite{CoSz12}. 
\begin{thm}
\begin{equation}\label{eq:Calabi-lb1}
 \underset{(\eta,\xi,\Phi,g) \in\cS(\xi,\ol{J})}{\inf}\Cal_{\xi,\ol{J}}(g)\geq c(n) \|\chi\|, 
\end{equation}
where $c(n)$ the constant, depending only on $n$, follows from the different scaling of the algebraic and analytic Futaki invariant.
The lower bound $\Cal_{\xi,\ol{J}}(g) =c(n) \|\chi\|$ is achieved if and only if $g$ is Sasaki-extremal.
\end{thm}
\begin{proof}
This is an application of the lower bound~\cite[Thm. 5]{CoSz12}
\begin{equation}\label{eq:Calabi-CoSz}
Cal_{\xi,\ol{J}}(g) \geq -c(n) \frac{\Fut_{\chi}(Y_0,\xi,\zeta)}{\|\zeta\|}, 
\end{equation}
for any $(\eta,\xi,\Phi,g) \in\cS(\xi,\ol{J})$ and test configuration $\calx$ with central fiber $Y_0$ and generator of $\C^*$ action $\zeta$.
Consider a product test configuration with $-\chi$ acting fiber-wise, in particular $\zeta =-\chi$.  The fact that $\chi$ is not necessarily rational is not 
a problem, because one can use an approximation argument.

For the second statement, if the lower bound is achieved at $(\eta,\xi,\Phi,g)$, then it is obviously a critical point
of $\Cal_{\xi,\ol{J}}^2$ and thus Sasaki-extremal.  Conversely, if $(\eta,\xi,\Phi,g)$ is Sasaki-extremal, then we can by conjugating
by an element of $\Aut(Y,\xi)$ assume that $T\subset\Aut (\eta,\xi,\Phi,g)$ is a maximal torus.   Then $S_g -\ol{S}_g$ is the normalized
holomorphy potential of $\chi$ and $\Cal_{\xi,\ol{J}}(g)^2 =c(n) ^2\|\chi\|^2$ be the equality, up to a constant, of the norm on $\gt$ defined in~\cite{FuMa95} and 
the algebraic norm.
\end{proof}

\begin{rmk}
This inequality follows easily from the arguments of~\cite{FuMa95} for $T$-invariant metrics for a maximal torus $T\subset\Aut(Y,\xi)$.  
But this extends the lower bound to all metrics in $\cS(\xi,\ol{J})$.
\end{rmk}

The following is the Sasakian version of a theorem in \cite{Sze06}:

\begin{thm}
Let $T\subset\Aut(Y,\xi)$ be a maximal torus, and let $\calx$ be a test configuration compatible with $T$ so that $\Fut_{\chi}(Y_0,\xi,\zeta) <0$.  Then
\begin{equation}\label{eq:Calabi-lb2}
\Cal_{\xi,\ol{J}}(g) ^2 \geq c(n)^2 \Bigl( \frac{\Fut_{\chi}(Y_0,\xi,\zeta)^2}{\|\zeta\|^2}  +\|\chi\|^2 \Bigr),
\end{equation}
for any $(\eta,\xi,\Phi,g) \in\cS(\xi,\ol{J})$.

In particular, for any $T$-invariant $(\eta,\xi,\Phi,g) \in\cS(\xi,\ol{J})$
\begin{equation*}
\Bigl(\int_M (S_g -\ol{S}_g  -h_g^{\chi})^2 \, d\mu_g \Bigr)^{1/2} \geq -c(n) \frac{\Fut_{\chi}(Y_0,\xi,\zeta)}{\|\zeta\|},
\end{equation*}
where $h_g^{\chi}$ is the normalized holomorphy potential for $\chi$ with respect to $g$.
\end{thm}
\begin{proof}
Choose a constant $\lambda$ so that $\ol{\zeta} =\zeta -\lambda\chi$ satisfies $\langle\ol{\zeta},\chi\rangle=0$.
Then $\Fut(Y_0,\xi,\ol{\zeta}) =\Fut_{\chi}(Y_0,\xi,\zeta) <0$.   Choose $\mu>0$ so that 
$\Fut(Y_0,\xi,\mu\ol{\zeta}) =-\|\mu\ol{\zeta}\|^2$.  If we define $\gamma=\mu\ol{\zeta} -\chi$, then
\[ \Fut(Y_0 ,\xi,\gamma)=-\|\mu\ol{\zeta}\|^2 -\|\chi\|^2 =-\|\gamma\|^2,\]
and
\[ \frac{\Fut(Y_0,\xi,\gamma)}{\|\gamma\|^2} =\frac{\Fut_{\chi}(Y_0,\xi,\zeta)}{\|\ol{\zeta}\|^2} +\|\chi\|^2 . \]
The inequality then follows from $\|\ol{\zeta}\|\leq\|\zeta\|$ and (\ref{eq:Calabi-CoSz}).
\end{proof}

The following easily follows from (\ref{eq:Calabi-lb2}).
\begin{cor}
If we have equality in (\ref{eq:Calabi-lb1}), in particular if there exists a Sasaki-extremal structure in $\cS(\xi,\ol{J})$, then 
$(Y,\xi)$ is K-semistable relative to any maximal torus $T$.
\end{cor}

\section{Modified Lichnerowicz Obstruction}

We will give a modified version of the Lichnerowicz obstruction of~\cite{GMSY06}, which will be useful in obstructing Sasaki-extremal metrics 
in examples constructed from hypersurface singularities.  

A necessary condition that  $(M,\eta,\xi,\Phi,g)$, with polarized cone $(Y,\xi)$, admits a Sasaki-Einstein metric is that $c_1(Y)=0$ (equivalently $c_1(\cald)=0$), and 
$o\in Y$ is a $\bbq$-Gorenstein singularity.   In fact, a stronger condition must hold.   One must have
\begin{equation}\label{eq:CY-con}
c_1(\calf_{\xi})=(n+1)[d\eta]_B.
\end{equation}
One can easily show, assuming $M$ is simply connected, that (\ref{eq:CY-con}) is equivalent to the existence 
of a non-vanishing $\Omega\in\Gamma\bigl( \Lambda^{n+1,0} C(M)\bigr)$ with 
\begin{equation}\label{cheqn}
\cL_\xi \Omega =\sqrt{-1}(n+1)\Omega.
\end{equation}
(cf.~\cite{FOW06}).   If $M$ is not simply connected, then $\Omega$ can be defined as multivalued.
Following \cite{MaSpYau06}  we define
\begin{equation}\label{chhyper}
\grS=\{\xi\in\gt^+~|~\cL_\xi \Omega =\sqrt{-1}(n+1)\Omega\}.
\end{equation}

We recall the Lichnerowicz obstruction of Gauntlett, Martelli, Sparks, and Yau.
\begin{thm}[\cite{GMSY06}]\label{thm:Lich}
Suppose $g$ is a holomorphic function on $Y$ of charge $0<\lambda <1$, i.e.
\[ \cL_{\xi} g=\sqrt{-1}\lambda g,\]
then $(Y,\xi)$ admits no Ricci-flat K\"{a}hler cone metric with Reeb vector field $\xi$.
\end{thm}

\begin{rem}\label{lamda1}
The Lichnerowicz obstruction was also refined by Obata to include $\grl=1$ in which case $(Y,\xi)$ must be the flat structure $(\bbc^{n+1},\xi_0)$. See Section 11.3.3 of \cite{BG05} and references therein.
\end{rem}

Suppose that $f$ is a weighted homogeneous holomorphic function with weight $\alpha\in\gt^*$.
Let $Y\cong Y_1\subset\Upsilon\rightarrow\C$ be the deformation to the normal cone of $V=\{f=0\}\subset Y$.  

\begin{thm}\label{thm:gen-Lich}
The following expression holds:
\begin{multline}\label{eq:DF1}
\Fut_{\chi}(Y_0,\xi,\zeta) = \\ \frac{a_0(\xi)}{2}(1-\frac{1}{\alpha(\xi)}) -\frac{1}{(n+2)(n+1)^2}\frac{ D_{\chi} a_0(\xi)}{\alpha(\xi)} -\frac{1}{(n+2)(n+1)}\frac{a_0(\xi)\alpha(\chi)}{\alpha(\xi)^2}. 
\end{multline}
\end{thm}

Notice that the righthand side is independent of the choice of $\chi$ modulo $\xi$.  If we choose $\chi\in\gt$ tangent to $\grS$, then this can be written as
\begin{multline}\label{eq:DF2}
\Fut_{\chi}(Y_0,\xi,\zeta) =\\ \frac{a_0(\xi)}{2}(1-\frac{1}{\alpha(\xi)}) -\frac{2}{(n+2)(n+1)^2}\frac{ \|\chi \|^2}{\alpha(\xi)} -\frac{1}{(n+2)(n+1)}\frac{a_0(\xi)\alpha(\chi)}{\alpha(\xi)^2} .
\end{multline}
\begin{proof}
It was shown in~\cite{CoSz12} that the Hilbert series of the central fiber $Y_0$ of $\Upsilon$ is 
\[ F(\xi+s\zeta) =\frac{a_0(\xi)\alpha(\xi)n!}{(\alpha(\xi)+s) t^{n+1}} +\frac{\alpha(\xi)(a_1(\xi)+\frac{sn}{2}a_0(\xi))(n-1)!}{(\alpha(\xi)+s) t^n} +O(t^{1-n}).\]
We denote the coefficients of the series as 
\[ \hat{a}_0(\xi +s\zeta)  =\frac{a_0(\xi)\alpha(\xi)}{\alpha(\xi) +s},\quad \hat{a}_1 (\xi +s\zeta)  =\frac{\alpha(\xi)\bigl(a_1(\xi) +\frac{sn}{2}a_0(\xi)  \bigr)}{ \alpha(\xi) +s}. \]
Then (\ref{eq:coef-der}) gives 
\[ \hat{b}_0 (\xi)=\frac{1}{n+1} \frac{a_0(\xi)}{\alpha(\xi)}, \quad \hat{b}_1(\xi) = \frac{1}{n}\bigl(\frac{a_1(\xi)}{\alpha(\xi)}-\frac{n}{2} a_0(\xi) \bigr).\]
We compute 
\[\begin{split}
\Fut(Y_0, \xi,\zeta)  & =\frac{1}{n+1}\frac{a_1(\xi)}{\alpha(\xi)} +\frac{1}{n}\bigl( \frac{n}{2} a_0(\xi) -\frac{a_1(\xi)}{\alpha(\xi)}\bigr) \\
			& = \frac{a_0(\xi)}{2}\bigl( 1-\frac{2}{n(n+1)} \frac{a_1(\xi)}{a_0(\xi) \alpha(\xi)} \bigr).
\end{split}\]
But the Calabi-Yau condition (\ref{eq:CY-con}) implies for $\xi\in\grS$ that $a_1(\xi) =\frac{n(n+1)}{2} a_0(\xi)$~\cite{CoSz12}, so we get 
\begin{equation}\label{eq:DF-nc}
\Fut(Y_0, \xi,\zeta) = \frac{a_0(\xi)}{2}\bigl( 1-\frac{1}{\alpha(\xi)}\bigr),
\end{equation}
which was originally obtained in~\cite{CoSz12}. 

Similarly, using (\ref{eq:coef-der}), we compute
\begin{equation}\label{eq:F-nc}
\begin{split}
\langle\zeta,\chi\rangle & = \frac{1}{(n+1)(n+2)} D_{\zeta} D_{\chi} \hat{a}_0 -\frac{1}{(n+1)^2} \frac{1}{a_0(\xi)} D_{\zeta} \hat{a}_0 D_{\chi} a_0 \\
				& = \frac{-1}{(n+1)(n+2)} D_{\chi} \Bigl( \frac{a_0(\xi)}{\alpha(\xi)}\Bigr) +\frac{1}{(n+1)^2} \frac{1}{\alpha(\xi)} D_{\chi} a_0(\xi) \\
				& = \frac{1}{(n+1)^2 (n+2)} \frac{D_{\chi} a_0(\xi)}{\alpha(\xi)} + \frac{1}{(n+1)(n+2)} \frac{a_0(\xi)\alpha(\chi)}{\alpha(\xi)^2},
\end{split}
\end{equation}
and \eqref{eq:DF2} follows from \eqref{eq:DF-nc} and \eqref{eq:F-nc}.  

To prove \eqref{eq:DF2} first observe that \eqref{eq:DF1} is independent of the choice of $\chi$ in $\Lie(T)/{\R\xi}$.  This follows easily from
$D_{\xi} a_0 =-(n+1) a_0(\xi)$.  It is well known that if $\chi$ is tangent to $\grS$ then $\frac{1}{2}D_{\chi} a_0(\xi) =\Fut(Y,\xi,\chi)$.   Thus by the definition
of $\chi$ we have $\frac{1}{2} D_{\chi} a_0(\xi) =\|\chi\|^2$, and \eqref{eq:DF2} follows.

\end{proof}

Since $\xi\in\grS$ Equation \eqref{eq:DF1} implies

\begin{cor}\label{cor:gen-Lich}
If $f$ is a homogeneous holomorphic function with weight $\alpha\in\gt^*$ satisfying the Lichnerowicz condition in Theorem~\ref{thm:Lich} and 
$\alpha|_{T\Sigma} =0$, then the entire Sasaki cone is obstructed from admitting extremal Sasaki metrics.
\end{cor}

\section{Applications to Links of Weighted Homogeneous Polynomials}
Let $f$ be a weighted homogeneous polynomial in $\bbc^{n+1}$ of degree $d$ and weight vector $\bfw=(w_0,\ldots,w_n)$ with only an isolated singularity at the origin. 
It is well known \cite{BG05} that the {\it link} of $f$ defined by $L_f=\{f=0\}\cap S^{2n+1}$ is an $(n-2)$-connected smooth manifold with a natural Sasakian structure $\cals_\bfw=(\xi_\bfw,\eta_\bfw,\Phi_\bfw,g_\bfw)$, that we call the {\it standard} Sasakian structure. 
Furthermore, according to \cite{BGK05} the Sasaki automorphism group will be finite if $2w_i<d$ for all but one $i=0,\ldots,n$. So we assume to the contrary that 
$2w_i\geq d$ for at least two of the $i=1,\ldots,n$. In fact, we shall assume that $f$ has the form
\begin{equation}\label{Feqn}
f(z_0,\ldots,z_n)=f'(z_0,\ldots,z_k)+z_{k+1}^2+\cdots +z_n^2
\end{equation}
with $n-k\geq 2$ and all weights $w_i$ with $i=0,\ldots,k$ satisfy $2w_i<d'$, the degree of $f'$. In this case the connected component of the Sasaki automorphism group is 
$U(1)\times SO(n-k)$. For convenience we order the weights 
$$w_0\leq w_1\leq \cdots\leq w_k.$$
We shall also assume that there is no linear factor as this implies that the link is a standard sphere. First we describe the Sasaki cone.

\begin{prop}\label{sasconeprop}
Let $f$ be a weighted homogeneous polynomial of the form of Equation \eqref{Feqn}. The Sasaki cone $\gt^+$ of the link $L_f$ is given by 
\begin{equation}\label{Fsascone}
\gt^+=\{b_0\xi_\bfw+\sum_{j=1}^rb_j\grz_j\in \gt~|~b_0>0,~-\frac{db_0}{4}<b_j<\frac{db_0}{4}\}
\end{equation}
where $\xi_\bfw$ is the standard Reeb field on $L_f$.
\end{prop}

\begin{proof}
First note that the dimension of the Sasaki cone $\gt^+$ is $r+1$ where $r=\lfloor\frac{n-k}{2}\rfloor$. We define variables 
$u_j=z_{k+j}+iz_{k+j+1}$ and $v_j=z_{k+j}-iz_{k+j+1}$ for $j=1,\ldots,r$ in which case $f$ takes the form
$$f(z_0,\ldots,z_n)=\begin{cases} 
                            f'(z_0,\ldots,z_k)+\sum_{j=1}^ru_jv_j, &\text{if $n-k$ is even}; \\
                            f'(z_0,\ldots,z_k)+\sum_{j=1}^ru_jv_j +z_n^2. & \text{if $n-k$ is odd.}
                            \end{cases}$$
Now the Sasaki cone $\gt^+$ is spanned by elements of the form 
\begin{equation}\label{tanform}
\xi=b_0\xi_\bfw+\sum_{j=1}^rb_j\grz_j
\end{equation}
where $\grz_j$ has weight $(1,-1)$ with respect to $(u_j,v_j)$ and $0$ elsewhere. Now it follows from Theorem 3 of \cite{CoSz12} that the Sasaki cone is determined by the domain of the smooth function $a_0(\xi)$ in the Laurant expansion of the index character $F(\xi,t)$ of Equation \eqref{a0eqn}. To find this we compute the Hilbert series for the ring $(R=\bbc[z_0,...,z_n])/I$ where $I$ is the ideal
generated by $f$, that is, $I=Rf$. See Proposition 4.3 of \cite{CoSz12}. The Hilbert series for $R$ is
$$\frac{1}{(1-e^{-w_{0}(\xi))t})}\cdots \frac{1}{(1- e^{-w_{n}(\xi))t})}$$ 
and that of $I$ is $1-e^{-W(\xi)t}$ where $W(\xi)$ is the weight of the polynomial $f$, that is, $W(\xi)=d$. This gives the Hilbert series of $R/I$ as 
\begin{equation}\label{Hevser}
F(\xi,t)=\frac{1-e^{-db_0t}}{\prod_{i=0}^k(1-e^{-w_ib_0t})\prod_{j=1}^r(1-e^{-(\frac{d}{2}b_0+b_j)t})(1-e^{-(\frac{d}{2}b_0-b_j)t})}
\end{equation}
if $n-k$ is even, and 
\begin{equation}\label{Hoddser}
F(\xi,t)=\frac{1-e^{-db_0t}}{\prod_{i=0}^k(1-e^{-w_ib_0t})\prod_{j=1}^r(1-e^{-(\frac{d}{2}b_0+b_j)t})(1-e^{-(\frac{d}{2}b_0-b_j)t})(1-e^{-\frac{d}{2}b_0t})}
\end{equation}
if $n-k$ is odd. Then using Equation \eqref{a0eqn} with $n\mapsto n-1$ we have for $n-k$ even
\begin{equation}\label{a0eqn2}
a_0(\xi)(n-1)!=\frac{d}{w_0\cdots w_kb_0^{k-1}(\frac{d^2b_0^2}{4}-b_1^2)\cdots (\frac{d^2b_0^2}{4}-b_r^2)}.
\end{equation}
So $a_0(\xi)$ is well defined and positive precisely for the range indicated. This finishes the proof when $n-k$ is even. There are similar expressions when $n-k$ is odd.
\end{proof}

\begin{thm}\label{whpobs}
Let $f$ be a weighted homogeneous polynomial of the form of Equation \eqref{Feqn} with no linear factors. Suppose further that the inequality holds
$$\sum_{i=0}^kw_i-w_0n+\frac{d}{2}(n-k-2)\geq 0.$$
Then there are no extremal Sasaki metrics in the entire Sasaki cone of the link $L_f$.
\end{thm}

\begin{proof}
We need to check that the conditions of Corollary \ref{cor:gen-Lich} are satisfied. To check the Lichnerowicz condition Theorem~\ref{thm:Lich} we consider the homogeneous 
function $g=z_0$. The Reeb vector field of the link $L_f$ is $\xi_\bfw$ which has Fano index $\cali=|\bfw|-d$. So the Reeb vector field which satisfies Equation \eqref{cheqn} is
$$\xi=\frac{n}{|\bfw|-d}\xi_\bfw.$$
Thus, the charge of $g$ is $\frac{nw_0}{|\bfw|-d}$. 
There are two cases to consider: (1) the degree $d'$ of $f'$ is even in which case $d=d'$, and (2) the degree $d'$ of $f'$ is odd and $d=2d'$. In both cases 
$w_{k+1}=\cdots =w_n=\frac{d}{2}$ and the Fano index $\cali$ is
\begin{equation}
\cali=|\bfw|-d=\sum_{i=0}^kw_i +\frac{d}{2}(n-k)-d=\sum_{i=0}^kw_i +\frac{d}{2}(n-k-2).
\end{equation}
Thus, the charge $\frac{nw_0}{|\bfw|-d}$ of $g$ will be less than one if and only if the strict inequality holds. Moreover, since there are no linear factors in $f$, Remark \ref{lamda1} also implies the result when equality holds.

Next, we check the condition $\gra |_{T\grS}=0$. Note first that $\xi\in \grS$ if and only if we have 
$$b_0=\frac{n}{\sum_{i=0}^kw_i +\frac{d}{2}(n-k-2)}<\frac{1}{w_0}$$
so the range on $\grS$ becomes
$$-\frac{1}{4}\frac{dn}{\sum_{i=0}^kw_i +\frac{d}{2}(n-k-2)}<b_j<\frac{1}{4}\frac{dn}{\sum_{i=0}^kw_i +\frac{d}{2}(n-k-2)}.$$
Now from \eqref{tanform} the tangent space to $\grS$ at $\xi$ is spanned by the set $\{\grz_j\}_{j=1}^r$, so if $\gra$ is the weight of $f$ then $\gra(\grz_j)=0$ since $(u_j,v_j)$ has weight $(1,-1)$ with respect to $\grz_j$.

\end{proof}

\section{Explicit Examples}
In this section our results are presented in the form of tables. The first table presents examples of families of Sasakian structures whose Sasaki cone has dimension greater than one,
and contains no extremal Sasaki metric at all. We shall only present details for a representative case. In the first four examples of Table \ref{sasconetable} the oriented diffeomorphism
type is stated explicitly where $\grS^{4n+1}_1$ is a generator  of the group $bP_{4n+2}$ of homotopy spheres which bound a parallelizable manifold of dimension $4n+2$. Note that 
$\grS^{4n+1}_k=\grS^{4n+1}_1$ if $(2k+1)\equiv \pm 3\mod 8$, and  is the standard when $(2k+1)\equiv \pm 1\mod 8$. It is known that 
$bP_{4n+2}\approx \bbz_2$ for $4n+2\neq 2^j-2$ for some $j$ and equals the identity if $n=1,3,7,15$, so in these cases $\grS^{4n+1}_1=\{id\}$. (It is still an open question for the
remaining cases.) Thus, in these four cases there are a countably infinite number of families of Sasakian structures with $n+1$ dimensional Sasaki cones having no extremal 
Sasaki metrics. Moreover, these belong to a countably infinite number of inequivalent underlying contact structures \cite{Ust99,KwvKo13,Gutt15,BMvK15,Ueb15}. 

The oriented diffeomorphism type can also be determined for the fifth example. These are homotopy spheres of dimension $4n-1$ which bound a parallelizable manifold of 
dimension $4n$. A well known result of Kervaire and Milnor \cite{KeMil63} says that the group $bP_{4n}$ is cyclic of order
$$|bP_{4n}|=2^{2n-2}(2^{2n-1}-1)~\hbox{numerator}~\!\!\biggl(\frac{4B_n}{n}\biggr),$$
where $B_n$ is the $n$-th Bernoulli number. Thus, for example $|bP_8|=28, |bP_{12}|=992, |bP_{16}|=8128.$ Thus, there are a countably infinite number of families of Sasakian
structures with $n$ dimensional Sasaki cones having no extremal Sasaki metrics. Moreover, these also belong to a countably infinite number of underlying contact structures 
(actually almost  contact structures in this case) since $\grS_{k'}^{4n-1}\equiv \grS_k^{4n-1}\mod |bP_{4n}|$, cf. Section 9.5.3 of \cite{BG05}. Here $\grS_1^{4n-1}$ is called 
the {\it Milnor generator}. 

The sixth example consists of a family of rational homology spheres of dimension $4n-1$ with $H_{2n}\approx \bbz_3$ described by Durfee \cite{Dur77}. There are two families of
oriented diffeomorphism types that are equivalent as non-oriented manifolds. We begin with the links $K_2$ and $K_4$ where $K_k$ is the link of the polynomial 
$z_0^k+z_1^3+z_2^2+\cdots z_{2n}^2$. Then $K_{6l+2}$ is diffeomorphic to $K_2\#(-1)^{\frac{n}{2}}l\grS_1^{4n-1}$ and $K_{6l+4}$ is diffeomorphic to 
$K_4\#(-1)^{\frac{n}{2}}l\grS_1^{4n-1}$. Note that the Milnor generator $\grS^{4n-1}_1\approx K_5$. Thus, in this case there are also a countably infinite number of families of 
Sasakian structures with $n$ dimensional Sasaki cones having no extremal Sasaki metrics. We conjecture that these also belong to inequivalent underlying contact structures, 
although this has not been proven yet to the best of the authors' knowledge.

In example seven we have Sasaki manifolds that are homeomorphic to $2k(S^{2n+1}\times S^{2n+2})$. The oriented diffeomorphism type is not known explicitly and not all possible
diffeomorphism types occur. In \cite{BG05h} a formula is given for the number of diffemorphism types $D_n(k)$ obtained by our method, and tables are given for dimension 7 and 11.
Nevertheless, since there is a periodicity modulo a subgroup of $bP_{4n+4}$, we do have countably infinite families of Sasakian structures with an $n+1$ dimensional Sasaki cones
having no extremal Sasaki metrics. Whether the underlying contact structures are inequivalent is also not known at this time.

Finally we consider the connected sums $\#m(S^2\times S^3)$ which are all the compact smooth Sasaki $5$-manifolds with a $2$-torus of Sasaki automorphisms. Here $m=0$ 
means $S^5$. These are represented by links of the polynomial $z_0^p+z_1^q+z_2^2+z_3^2$ with $m=\gcd(p,q)$. They are perhaps the most interesting case owing to the recent 
work of Collins and Sz\'ekelyhidi \cite{CoSz15} which shows that the standard Sasakian structure on the link admits a Sasaki-Einstein metric if and only if $2p>q$ and $2q>p$. 
Thus, when these inequalities are violated, namely when $p\geq 2q$ or $q\geq 2p$, there are for each $m=0,1,\cdots$ a countably infinite number of families of Sasakian structures 
with $2$-dimensional Sasaki cones having no extremal Sasaki metrics. Moreover, they belong to a countably infinite number of inequivalent underlying contact structures 
~\cite{BMvK15,Ueb15}. 

\begin{question}
In the case that Sasaki-Einstein metrics exist, that is when $2p>q$ and $2q>p$, is the entire Sasaki cone exhausted by extremal Sasaki metrics?
\end{question}

Although in most cases of Table \ref{sasconetable} the members of a given family belong to inequivalent contact structures, they do have isomorphic transverse holomorphic structures.
This can be seen by showing that the $S^1$ quotients of the link for each member are isomorphic as algebraic varieties. We refer to the proof of Proposition 4.1 of \cite{BGN03b} for the details in the case of the homotopy spheres. This completes the proof of Theorem \ref{infcon}. \hfill $\Box$

\begin{table}[h]
$\begin{array}{|c|c|c|}
\hline
\text{Diffeo-(homeo)-morphism Type}        & f &  \dim \gt^+ \\\hline
S^{2n}\times S^{2n+1}  & z_0^{8l}+z_1^2+\cdots +z_{2n+1}^2,\ n,l\geq1   &  n+1 \\
S^{2n}\times S^{2n+1}\#\grS^{4n+1}_1 & z_0^{8l+4}+z_1^2+\cdots +z_{2n+1}^2=0,\ n\geq1,l\geq 0  & n+1 \\
\text{Unit tangent bundle of $S^{2n+1}$}  & z_0^{4l+2}+z_1^2+\cdots +z_{2n+1}^2, \ n>1,l\geq 1 &  n+1 \\ 
\text{Homotopy sphere}~\grS_k^{4n+1}  & z_0^{2k+1}+  z_1^2 +\cdots +z_{2n+1}^2, \ n>1, k\geq 1 &  n+1 \\
\text{Homotopy sphere}~\grS_k^{4n-1}  & z_0^{6k-1}+  z_1^3 +z_2^2 +\cdots +z_{2n}^2, \ n\geq 2,k\geq 1 &  n \\
\text{Rat. homology sphere}~H_{2n}\approx\bbz_3  & z_0^{k}+  z_1^3 +\cdots +z_{2n}^2,\ n,k>1 &  n \\
2k(S^{2n+1}\times S^{2n+2}),~~D_{n+1}(k)    & z_0^{2(2k+1)} +z_1^{2k+1}+ z_2^2 +\cdots +z_{2n+2}^2, \ n,k\geq 1 & n+1 \\
\#m(S^2\times S^3), m=\gcd(p,,q)-1    &  z_0^p+z_1^q+z_2^2+z_3^2,~ \ p\geq 2q~\text{or}~q\geq 2p  &   2 \\
\hline
\end{array}$
\vspace{.1in}
\caption{Manifolds having Sasaki Cones with no Extremal Metrics}
\label{sasconetable}
\end{table}

We give the simple details of applying Theorem \ref{whpobs} for a sample case. Consider the unit tangent sphere bundle $T$ of $S^{2n+1}$ represented by the link 
$L(4l+2,2,\ldots,2)$ given by the polynomial
$$f=z_0^{4l+2}+z_1^2+\cdots +z_{2n+1}^2.$$
We mention that there are no exotic structures in this case since $T\# \grS^{4n+1}$ is diffeomorphic to $T$. Applying Theorem \ref{whpobs} we take $n\mapsto 2n+1$, 
$k=0,d=4l+2$ and $w_0=1$, so our inequality becomes
$$0<w_0-w_0(2n+1)+(2l+1)(2n-1)=2l(2n-1)-1.$$

Of course there are many  other examples to which we can apply Theorem \ref{whpobs}. For example, the rational homology spheres $M_k^{4n-1}$ of dimension $4n-1$ with
$H_{2n-1}\approx\bbz_k$  given by Briskorn-Pham link of the polynomial  $z_0^{k}+  z_1^2 +\cdots +z_{2n}^2, $ with $k>2,n>1$ has a Sasaki cone of dimension  $n+1$ and 
one can easily check that the inequality of Theorem \ref{whpobs} holds in this case. So for each $k>2$ $M_k^{4n-1}$ has an $n+1$-dimensional Sasaki cone having no 
extremal Sasaki metrics. We mention that the case $k=2$ $M_2^{4n-1}$ is a Stiefel manifold which admits a Sasaki-Einstein metric. Many more examples of 
rational homology spheres can be obtained by noting that as in Corollary 9.5.3 of \cite{BG05} if the link of $f'$ in the polynomial \eqref{Feqn} is a rational homology sphere and 
$n-k$ is even, then the link of $f$ is also a rational homology sphere.

Finally we give a table of ADE $n$-folds which admit no extremal Sasaki metrics in the entire Sasaki cone or equivalently no Ricci-flat K\"ahler metric on the corresponding Calabi-Yau cone. Notice that there is some overlap with Table \ref{sasconetable}. These are of particular interest since they have been used in Physics in conformal field theory \cite{GVW00,GMSY06}. 

\begin{table}[h]
$\begin{array}{|c|c|c|c|}
\hline
          & f & \mathbf{w} &  \dim \gt^+  \\\hline
 A_{k-1}  & z_0^k +z_1^2 +\cdots +z_n^2,\ k\geq 3   & (2,k,\cdots,k) & 1+\lfloor\frac{n}{2}\rfloor  \\
 D_{k+1} & z_0^k +z_0 z_1^2 +z_2^2 +\cdots +z_n^2 ,\ k\geq 2 & (2,k-1,k,\cdots,k) & \lceil\frac{n}{2}\rceil \\
 E_6        & z_0^4 +z_1^3 + z_2^2 +\cdots +z_n^2 & (4,3,6,\cdots,6)  & \lceil\frac{n}{2}\rceil \\ 
 E_7        & z_0^3 +z_0 z_1^3 +z_2^2 +\cdots +z_n^2 & (6,4,9,\cdots,9) & \lceil\frac{n}{2}\rceil \\
 E_8        &  z_0^5 +z_1^3 +z_2^2 +\cdots +z_n^2 & (6,10,15,\cdots,15) & \lceil\frac{n}{2}\rceil \\
 \hline
\end{array}$
\vspace{.1in}
\caption{ADE n-folds with $n\geq 4$ whose Sasaki Cones have no Extremal Metrics}
\label{fig:4-fold}
\end{table}

\def\cprime{$'$} \def\cprime{$'$} \def\cprime{$'$} \def\cprime{$'$}
  \def\cprime{$'$} \def\cprime{$'$} \def\cprime{$'$} \def\cprime{$'$}
  \def\cdprime{$''$} \def\cprime{$'$} \def\cprime{$'$} \def\cprime{$'$}
  \def\cprime{$'$}
\providecommand{\bysame}{\leavevmode\hbox to3em{\hrulefill}\thinspace}
\providecommand{\MR}{\relax\ifhmode\unskip\space\fi MR }
\providecommand{\MRhref}[2]{%
  \href{http://www.ams.org/mathscinet-getitem?mr=#1}{#2}
}
\providecommand{\href}[2]{#2}

\end{document}